\documentclass[12pt]{amsart}
\pdfoutput=1
\usepackage{microtype}

\usepackage{amssymb}
\usepackage{amsthm}
\usepackage{mathtools}
\usepackage{tikz-cd}
\usepackage{color}
\usepackage{hyperref}
\usepackage{upref}
\usepackage{enumitem}
\usepackage{verbatim}

\usepackage[cmtip,all]{xy}

\allowdisplaybreaks

\newtheorem{thm}{Theorem}[section]

\newtheorem{prop}[thm]{Proposition}
\newtheorem{lem}[thm]{Lemma}

\theoremstyle{definition}
\newtheorem{defn}[thm]{Definition}

\theoremstyle{remark}
\newtheorem{rem}[thm]{Remark}

\newcommand{\thmref}[1]{Theorem~\textup{\ref{#1}}}

\newcommand{\secref}[1]{Section~\textup{\ref{#1}}}
\newcommand{\lemref}[1]{Lemma~\textup{\ref{#1}}}
\newcommand{\defnref}[1]{Definition~\textup{\ref{#1}}}

\numberwithin{equation}{section}

\renewcommand{\AA}{\mathcal A}

\newcommand{\CC}{\mathcal C}

\newcommand{\wilde}{\widetilde}

\newcommand{\xt}{\otimes}
\newcommand{\what}{\widehat}
\newcommand{\id}{\text{\textup{id}}}

\newcommand{\variso}{\xrightarrow{\simeq}}
\newcommand{\midtext}[1]{\quad\text{#1}\quad}
\newcommand{\righttext}[1]{\qquad\text{#1 }}

\newcommand{\cstg}{\ensuremath{C^*(G)}}

\newcommand{\dg}{\ensuremath{\delta_G}}

\newcommand{\rt}{\textup{rt}}

\newcommand{\ad}{\operatorname{Ad}}
\renewcommand{\bar}{\overline}
\newcommand{\spn}{\operatorname{span}}
\newcommand{\clspn}{\bar{\spn}}
\newcommand{\inv}{^{-1}}
\renewcommand{\subset}{\subseteq}

\renewcommand{\epsilon}{\varepsilon}

\definecolor{alizarin}{rgb}{0.82, 0.1, 0.26}
\definecolor{blue-violet}{rgb}{0.54, 0.17, 0.89}

\newcommand{\coo}{\mathcal{C}}
\newcommand{\waco}{\mathcal{A}}
\newcommand{\wcpo}{\operatorname{CP}}
\DeclareMathOperator{\mor}{Mor}

\begin{document}

\title{Strong Pedersen rigidity for coactions of compact groups}
\begin{abstract}
We prove a version of Pedersen's outer conjugacy theorem for coactions of compact groups, which characterizes outer conjugate coactions of a compact group in terms of properties of the dual actions.
In fact, we show that every isomorphism of a dual action comes from a unique outer conjugacy of a coaction, which in this context should be called \emph{strong Pedersen rigidity}.
We promote this to a category equivalence.
\end{abstract}

\author[Kaliszewski]{S. Kaliszewski}
\address{School of Mathematical and Statistical Sciences, Arizona State University, Tempe, AZ 85287}
\email{kaliszewski@asu.edu}
\author[Omland]{Tron Omland}
\address{Norwegian National Security Authority (NSM);
Department of Mathematics\\University of Oslo\\Norway}
\email{tron.omland@gmail.com}
\author[Quigg]{John Quigg}
\address{School of Mathematical and Statistical Sciences, Arizona State University, Tempe, AZ 85287}
\email{quigg@asu.edu}
\author[Turk]{Jonathan Turk}
\address{School of Mathematical and Statistical Sciences, Arizona State University, Tempe, AZ 85287}
\email{jturk@asu.edu}

\date{September 28, 2023}


\maketitle

\section{Introduction}

Given an action $\alpha$ of a locally compact group $G$ on a $C^*$-algebra $A$,
one forms the crossed product $C^*$-algebra $A \rtimes_\alpha G$.
The question is then how much the crossed product remembers of the underlying dynamics,
assuming knowledge of the group $G$.
In general, if we only know the crossed product up to isomorphism,
we cannot say much about the underlying dynamics,
not even up to Morita equivalence.
Thus, we need some extra information about the crossed product,
chosen with an eye toward recovering the action up to some type of equivalence,
and this concept is often thought of as the rigidity of the dynamical system.

Pedersen \cite[Theorem~35]{pedersen} proved that two actions $(A,\alpha)$ and $(B,\beta)$ of a locally compact abelian group $G$ 
are outer conjugate
(he used the term ``exterior equivalent'', which is inconsistent with current use)
if and only if there is an isomorphism between the crossed products that is equivariant for the dual actions
$\what\alpha$ and $\what\beta$
and respects the embeddings of $A$ and $B$.
We call the assumption on respecting the embeddings \emph{Pedersen's condition}.

In \cite[Theorem~3.1]{three}, the first three authors (to whom we refer here as ``KOQ'')
extended Pedersen's theorem to nonabelian groups,
using dual coactions,
and in \cite[Theorem~5.9]{three}
gave a categorical context for Pedersen's theorem,
establishing an equivalence between an \emph{outer category} of actions
and a \emph{fixed-point equivariant category} of coactions.

KOQ ran into difficulty when attempting to find examples
showing that Pedersen's condition on the embeddings is necessary.
This would mean finding two actions of $G$ that are not outer conjugate but nevertheless have isomorphic dual coactions.
Not only were no counterexamples found, KOQ proved that no such examples exist under a variety of circumstances, for instance (just to give a representative sampling) $G$ discrete, $A$ and $B$ stable or commutative, or $\alpha$ and $\beta$ inner.
KOQ referred to such results as ``no-go theorems''.
More accurately, in \cite{rigid} the aforementioned no-go theorems were proved for abelian groups,
and in \cite{rigid2} this extra assumption was dropped.

KOQ called situations in which Pedersen's condition is redundant \emph{Pedersen rigid} (for actions) and \emph{fixed-point rigid} (for the equivariant coactions). And this is implied by formally stronger properties called strongly Pedersen rigid and strongly fixed-point rigid, respectively.
More precisely, a coaction $(A,\delta)$
was called strongly fixed-point rigid
if it has a unique generalized fixed-point algebra, independent of the choice of $\delta_G-\delta$ equivariant unitary homomorphism $V:\cstg\to M(A)$,
and an action $(A,\alpha)$ was called strongly Pedersen rigid if its dual coaction is strongly fixed-point rigid.
We now regard these definitions to be flawed.
See the end of \secref{prelim} for corrected versions.

In view of the failure to find counterexamples, there still remains what KOQ called the \emph{strong Pedersen rigidity problem}:
Is every action of $G$ strongly Pedersen rigid?
Equivalently, is every equivariant coaction of $G$ strongly fixed-point rigid?
(There were also slightly weaker versions, without the adjective ``strong''.)

In this paper we begin the investigation of the dual problem,
involving coactions and their dual actions.
The first roadblock is that we need a fully working version of Pedersen's theorem,
which can be formulated by routinely dualizing everything,
and would say:
``two coactions $(A,\delta)$ and $(B,\epsilon)$ of $G$ are outer conjugate
if and only if there is an isomorphism between the crossed products
that is equivariant for the dual actions $\what\delta$ and $\what\epsilon$
and respects the embeddings of $A$ and $B$.''
As KOQ point out in \cite{three}, one direction of this desired theorem
is essentially recorded in \cite[Proposition~2.8]{twisted},
where it was phrased for exterior equivalence, but a routine 
modification shows that
if the coactions are outer conjugate then there is an isomorphism of the dual actions that preserves the embeddings.
But in the general case the converse is still open.
In the present paper, specializing to compact groups, we prove the converse direction of this Pedersen's theorem,
which ultimately allows us to prove a no-go theorem
and a category equivalence.

The outer category of actions has also been considered in \cite[Section~1.3]{szabo}, where it is called the cocycle category,
and where isomorphisms are called cocycle conjugacies.
Pedersen rigidity for actions of discrete groups on commutative $C^*$-algebras was treated with a different approach,
using groupoids, and just called $C^*$-rigidity in \cite[Remark~7.6]{CRST}.

We thank the anonymous referee for comments that greatly improved our paper.

\section{Preliminaries}\label{prelim}

We begin by recalling some terminology from \cite{rigid, rigid2}.
However, although in \cite{rigid2} $G$ was allowed to be any locally compact group,
here we require at least that $G$ be amenable ---
the main results of this paper will actually have $G$ compact,
and the formalism from \cite{rigid2} is more clearly stated for amenable groups
(so that, for example, all coactions are both maximal and normal).

Let $(A,\delta)$ be a coaction of $G$.
The \emph{crossed product} is given by the triple $(A\rtimes_\delta G,\what\delta,j_G)$,
where $j_G$ is part of the universal covariant homomorphism $(j_A,j_G)$ of $(A,C_0(G))$ into $M(A\rtimes_\delta G)$ and $\what\delta$ is the \emph{dual action} on $A\rtimes_\delta G$ given by
$\what\delta_s=j_A\times (j_G\circ\rt_s)$,
where $\rt$ is the action by right translation of $G$ on $C_0(G)$.

Given two coactions $(A,\delta)$ and $(B,\epsilon)$,
a $\delta-\epsilon$ equivariant
nondegenerate homomorphism $\phi:A\to M(B)$
induces a $\what\delta-\what\epsilon$ equivariant nondegenerate homomorphism
$\phi\rtimes G:A\rtimes_\delta G\to M(B\rtimes_\epsilon G)$,
and $(A,\delta)$ and $(B,\epsilon)$ are \emph{isomorphic} if there is
a $\delta-\epsilon$ equivariant isomorphism $A\variso B$.

An \emph{equivariant action} is a triple $(A,\alpha,\mu)$, where
$(A,\alpha)$ is an action and $\mu:C_0(G)\to M(A)$ is a nondegenerate $\rt-\alpha$ equivariant homomorphism.
Landstad duality for coactions \cite[Theorem~3.3 and Proposition~3.2]{ldc}
says that there is a coaction $(B,\delta)$,
unique up to isomorphism,
and a $\what\delta-\alpha$ equivariant isomorphism $\Theta:B\rtimes_\delta G\variso A$
such that $\Theta\circ j_G=\mu$.
Consequently, the subalgebra $j_B(B)$ of $M(A)$ depends only upon $\alpha$ and $\mu$,
and is called the \emph{generalized fixed-point algebra}, which we denote by $A^{\alpha,\mu}$.
Since $G$ is amenable, the homomorphism $j_B$ is an isomorphism of $B$ onto $A^{\alpha,\mu}$.
If we use this to identify $B$ with the generalized fixed-point algebra then $\delta$ becomes
the restriction to $B$ of (the canonical extension to $M(A)$ of)
the inner coaction $\delta^\mu$ of $G$ on $A$ implemented by $\mu$
(and then of course $j_G$ becomes $\mu$).
Again since we are assuming $G$ to be amenable here,
\cite[Theorem~4.3]{ldc} says that $A^{\alpha,\mu}$ coincides with the set of all $a\in M(A)$
satisfying
\emph{Landstad's coconditions} \cite[Definition~4.1]{ldc}:
\begin{enumerate}[label=(L\arabic*), nosep]
\item\label{L1}
$\alpha_s(a)=a$ for all $s\in G$;
\item\label{L2}
$a\mu(C_0(G))\cup\mu(C_0(G))a\subset A$;
\item\label{L3}
$\delta^\mu(a)\in \wilde M(M(A)\xt C^*(G))$.
\end{enumerate}
The right-hand side of \ref{L3}
refers to the multipliers $m$ of $M(A)\xt\cstg$
such that $m(1\xt\cstg)\cup (1\xt\cstg)m\subset M(A)\xt\cstg$.

Dually, we denote the crossed product of an action $(A,\alpha)$ by
$(A\rtimes_\alpha G,\what\alpha,i_G)$, where $i_G$ is part of the universal covariant homomorphism $(i_A,i_G)$ of $(A,C^*(G))$, and $\what\alpha$ is the dual coaction on $A\rtimes_\alpha G$ determined by
\begin{align*}
\what\alpha\circ i_A(a)&=i_A(a)\xt 1\righttext{for}a\in A\\
\what\alpha\circ i_G(s)&=i_G(s)\xt s\righttext{for}s\in G.
\end{align*}
If $(B,\beta)$ is another action, an $\alpha-\beta$ equivariant nondegenerate homomorphism $\phi:A\to  M(B)$ induces an $\what\alpha-\what\beta$ equivariant nondegenerate homomorphism $\phi\rtimes G:A\rtimes_\alpha G\to M(B\rtimes_\beta G)$,
and $(A,\alpha)$ and $(B,\beta)$ are \emph{isomorphic} if there is an
$\alpha-\beta$ equivariant isomorphism $A\simeq B$.
An \emph{equivariant coaction} is a triple $(A,\delta,U)$, where $(A,\delta)$ is a coaction and
$U:G\to M(A)$ is a strictly continuous unitary homomorphism
that is $G$-equivariant in the sense that
$\delta(U_s)=U_s\xt s$ for $s\in G$.
Landstad duality for coactions
says that there is an action $(B,\alpha)$, unique up to isomorphism,
and an isomorphism $\Theta:(B\rtimes_\alpha G,\what\alpha)\variso (A,\delta)$
such that $\Theta\circ i_G=U$.
Consequently, the subalgebra $i_B(B)$ of $M(A)$ depends only upon $\delta$ and $U$, and is called the \emph{generalized fixed-point algebra}, which we denote by $A^{\delta,U}$.

Given a coaction $(A,\delta)$,
a \emph{$\delta$-cocycle} is a unitary  $U\in M(A\xt C^*(G))$
such that
\begin{enumerate}[label=(\arabic*), nosep]
\item
$(\id\xt\delta_G)(U)=(U\xt 1)(\delta\circ\id)(U)$ and
\item
$\ad U\circ\delta(A)(1\xt C^*(G))\subset A\xt C^*(G)$.
\end{enumerate}
The \emph{perturbed coaction} is $\ad U\circ\delta$,
also called a \emph{perturbation} of $\delta$,
and is 
by definition \emph{exterior equivalent} to $\delta$.
Then $(A,\delta)$ is \emph{outer conjugate} to another coaction $(B,\epsilon)$ if it is isomorphic to a perturbation of $(B,\epsilon)$.

Dually, given an action $(A,\alpha)$, an \emph{$\alpha$-cocycle} is a
strictly continuous unitary map $u:G\to M(A)$ such that
$u_{st}=u_s\alpha_s(u_t)$ for $s,t\in G$.
The \emph{perturbed action} is $(\ad u\circ\alpha)_s=\ad u_s\circ\alpha_s$,
also called a \emph{perturbation} of $\alpha$, and is by definition 
\emph{exterior equivalent} fo $\alpha$.
Then $(A,\alpha)$ is \emph{outer conjugate} to another action $(B,\beta)$ if it is isomorphic to a perturbation of $(B,\beta)$.

An action $(A,\alpha)$ is \emph{Pedersen rigid} if for every action $(B,\beta)$,
$\alpha$ is outer conjugate to $\beta$ if and only if
the dual coactions $\what\alpha$ and $\what\beta$ are isomorphic.
Dually, a coaction $(A,\delta)$ is \emph{Pedersen rigid} if for every coaction $(B,\epsilon)$,
$\delta$ is outer conjugate to $\epsilon$ if and only if
the dual actions $\what\delta$ and $\what\epsilon$ are isomorphic.

There are ``strong'' versions:
an action $(A,\alpha)$ is \emph{strongly Pedersen rigid} if for every action $(B,\beta)$,
if $\Theta:(A\rtimes_\alpha G,\what\alpha)\variso (B\rtimes_\beta G,\what\beta)$
is an isomorphism, then $\Theta(i_A(A))=i_B(B)$,
and dually
a coaction $(A,\delta)$ is \emph{strongly Pedersen rigid} if for every coaction $(B,\beta)$,
if $\Theta:(A\rtimes_\delta G,\what\delta)\variso (B\rtimes_\epsilon G,\what\epsilon)$
is an isomorphism, then $\Theta(j_A(A))=j_B(B)$.

The \emph{Pedersen rigidity problem} for actions is whether every action is Pedersen rigid,
and similarly for the \emph{strong Pedersen rigidity problem};
in this paper we are interested in the corresponding problems for coactions,
and we give an affirmative answer when $G$ is compact
(\thmref{no-go}).

Pedersen's theorem for actions says that two actions $\alpha$ and $\beta$ of $G$ on $A$ are
exterior equivalent if and only if
there is an isomorphism
$\Theta:(A\rtimes_\alpha G,\what\alpha)\variso (A\rtimes_\beta G,\what\beta)$
such that
\[
\Theta\circ j_A^\alpha=j_A^\beta,
\]
and actions $(A,\alpha)$ and $(B,\beta)$ of $G$ are
outer conjugate if and only if
there is an isomorphism
$\Theta:(A\rtimes_\alpha G,\what\alpha)\variso (B\rtimes_\beta G,\what\beta)$
such that $\Theta(i_A(A))=i_B(B)$.
There is at present no dual version (with actions replaced by coactions)
for arbitrary $G$,
but in this paper we will prove a dual version
(\thmref{pedersen})
under the assumption that $G$ is compact.

Pedersen rigidity (and its strong counterpart) for actions can be equivalently expressed in terms of the dual coactions, using Landstad duality and Pedersen's theorem.

The following definitions have been modified and augmented from \cite{rigid,rigid2}:

\begin{defn}
\begin{enumerate}
\item
A coaction $(A,\delta)$ is \emph{dual} if
there exists an action $(B,\alpha)$ such that $(A,\delta)\simeq (B\rtimes_\alpha G,\what\alpha)$.

\item
A dual coaction $(A,\delta)$ is
\emph{strongly fixed-point rigid} if
the generalized fixed-point algebra $A^{\delta,V}$ is independent of the choice of equivariant
strictly continuous unitary homomorphism $V:G\to M(A)$,
i.e., for any two such $V,W:G\to M(A)$ we have
$A^{\delta,V}=A^{\delta,W}$.

\item
A dual coaction $(A,\delta)$ is
\emph{fixed-point rigid} if
for any two $G$-equivariant strictly continuous unitary homomorphisms $V,W:G\to M(A)$
there is an automorphism $\Theta$ of $(A,\delta)$ such that
$\Theta(A^{\delta,V})=A^{\delta,W}$.
\end{enumerate}
\end{defn}

Note that dual coactions are automatically maximal.
Also, by Landstad duality a coaction $(A,\delta)$ is dual if (and only if)
there exists a $G$-equivariant strictly continuous unitary
homomorphism $V:G\to M(A)$.

\begin{rem}
In \cite{rigid2} we defined fixed-point rigidity (strong or not) without worrying about whether there exists any such $V$. We now regard this to be flawed: if no such $V$ exists, then the coaction would be called strongly fixed-point rigid by default.
The above notion of dual coaction is intended to rectify this.
Although we do not need it here, we point out that we would now make a similar change to the definition of \emph{fixed-point rigid actions} (strong or not), restricting it to \emph{dual actions}
(defined below).
\end{rem}

\begin{defn}\label{dual}
An action $(C,\alpha)$ is \emph{dual} if there exists a coaction $(A,\delta)$ such that $(A\rtimes_\delta G,\what\delta)\simeq (C,\alpha)$.
\end{defn}

\begin{rem}
By Landstad duality and Pedersen's theorem,
actions are Pedersen rigid if and only if their dual coactions are fixed-point rigid,
and similarly for the strong versions.
A dual version, namely that coactions are Pedersen rigid if and only if their dual actions are fixed-point rigid,
and similarly for the strong versions, would require a fully working version of Pedersen's theorem for coactions\footnote{although Landstad duality is no problem}.
By our \thmref{pedersen}, at present we can only deduce this equivalence when $G$ is compact.
\end{rem}

\section{Reduction of coconditions for compact groups}\label{reduction}

The following lemma shows that, when $G$ is compact, Landstad's coconditions can be greatly simplified.

\begin{lem}\label{fixed}
Let $(A,\alpha,\mu)$ be an equivariant action.
When $G$ is compact,
Landstad's coconditions \ref{L2}--\ref{L3}
are redundant,
so that $A^{\alpha,\mu}$ coincides with the fixed-point algebra $A^\alpha$.
\end{lem}

\begin{proof}
First note that, by Landstad duality, we have
\[
A=\clspn\{A^{\alpha,\mu}\mu(C(G))\},
\]
and because $C(G)$ is unital we must have $A^{\alpha,\mu}\subset A$.
Thus Landstad's cocondition~\ref{L1}
implies that $A^{\alpha,\mu}\subset A^\alpha$,
and for the opposite containment
it suffices to show that if $a\in A^\alpha$ then $a$ satisfies \ref{L2}--\ref{L3}.
Since 
$\mu(C(G))\subset M(A)$
and $a\in A$,
\ref{L2}
is trivially satisfied.

For the other part, the homomorphism $\mu$ implements an inner coaction $\delta$ of $G$ on $A$, so
if $c\in C^*(G)$ then,
because $a\in A$,
we have $\delta(a)(1\xt c)\in A\xt C^*(G)$.
Since $C^*(G)$ is nuclear (because $G$ is amenable) the tensor product $A^\alpha\xt C^*(G)$ is canonically identified with a $C^*$-subalgebra of $A\xt C^*(G)$.
In fact, we have
\[
A^\alpha\xt C^*(G)=\bigl(A\xt C^*(G)\bigr)^{\alpha\xt\id}.
\]
Since $a\in A^\alpha$,
for all $s\in G$
we have
\begin{align*}
(\alpha_s\xt\id)\bigl(\delta(a)(1\xt c)\bigr)
&=\delta\bigl(\alpha_s(a)\bigr)(1\xt c)
\\&=\delta(a)(1\xt c),
\end{align*}
so
\[
\delta(a)(1\xt c)\in A^\alpha\xt C^*(G).
\]
Similarly for $(1\xt c)\delta(a)$,
and we have verified \ref{L3}.
\end{proof}

We apply \lemref{fixed} to prove our no-go theorem:
\begin{thm}\label{no-go}
If $G$ is compact, then every coaction of $G$ is strongly Pedersen rigid.
\end{thm}

\begin{proof}
By Landstad duality, it suffices to show that every equivariant action of $G$ is strongly fixed-point rigid.
But this follows immediately from \lemref{fixed}
\end{proof}

\section{Pedersen's theorem for coactions of compact groups}\label{sec:pedersen}

Recall that $G$ denotes a compact group.

\begin{thm}\label{pedersen}
\begin{enumerate}[label=\textup{(\arabic*)}]
\item
Two coactions $\delta$ and $\epsilon$ of $G$ on $A$ are
exterior equivalent if and only if
there is an isomorphism
$\Theta:(A\rtimes_\delta G,\what\delta)\variso (A\rtimes_\epsilon G,\what\epsilon)$
such that
\begin{equation}\label{theta iso}
\Theta\circ j_A^\delta=j_A^\epsilon.
\end{equation}
In fact, if $U$ is an $\epsilon$-cocycle such that $\delta=\ad U\circ\epsilon$,
then there is an isomorphism
$\Phi_U:(A\rtimes_\delta G,\what\delta)\variso (A\rtimes_\epsilon G,\what\epsilon)$
satisfying \eqref{theta iso} such that
\begin{equation}\label{theta G}
(\Phi_U\circ j_G^\delta\xt\id)(w_G)=(j_A^\epsilon\xt\id)(U)(j_G^\epsilon\xt\id)(w_G),
\end{equation}
and conversely every isomorphism $\Theta$ satisfying Equations\ref{theta iso} and \ref{theta G}
is of the form $\Phi_U$ for a unique $\epsilon$-cocycle $U$.

\item
Two coactions $(A,\delta)$ and $(B,\epsilon)$ of $G$ are
outer conjugate if and only if
there is an isomorphism
$\Theta:(A\rtimes_\delta G,\what\delta)\variso (B\rtimes_\epsilon G,\what\epsilon)$.
\end{enumerate}
\end{thm}

Note that in (2) we did not explicitly mention Pedersen's condition
$\Theta(j_A(A))=j_B(B)$, since that follows automatically from
\lemref{fixed} because, for example,
$j_A(A)=(A\rtimes_\delta G)^{\what\delta,j_G^\delta}$ by Landstad duality.

\begin{proof}
(1)
Three of us proved in
\cite[Proposition~3.6]{three},
(based on
\cite[Proposition~2.8 (i)]{twisted})
that for every $\epsilon$-cocycle $U$ there is an
isomorphism $\Phi_U:(A\rtimes_\delta G,\what\delta)\variso (A\rtimes_\epsilon G,\what\epsilon)$
satisfying \ref{theta iso} and \ref{theta G}.
Conversely,  
assume that we have an isomorphism
$\Theta:(A\rtimes_\delta G,\what\delta)
\variso
(A\rtimes_\epsilon G,\what\epsilon)$
taking $j_A^\delta$ to $j_A^\epsilon$.
Use this to identify
$(A\rtimes_\delta G,\what\delta)
=(A\rtimes_\epsilon G,\what\epsilon)
=(A\rtimes G,\alpha)$
and
$j_A^\delta=j_A^\epsilon=j$.
Our group $G$ is compact, and hence amenable, so $j$ is injective.
Replace $A$ by $j(A)$.
Note that $\Theta$ is now an automorphism of $A\rtimes G$,
and $\Theta|_A=\id_A=j_A^\epsilon$.

Let $\mu=\Theta\circ j_G^\delta$ and $\nu=j_G^\epsilon$,
let the associated corepresentations
(see, for example, \cite[Lemma~1.2]{twisted} and the discussion preceding it)
be denoted by
\[
V=\mu\xt\id(w_G)
\midtext{and}
W=\nu\xt\id(w_G),
\]
where $w_G\in M(C_0(G)\xt C^*(G))=C_b^\beta(G,M(C^*(G)))$
(strictly continuous norm-bounded functions)
is defined by $w_G(s)=s$,
and then put
\[
U=VW^*.
\]
We will show that $U$ is an $\epsilon$-cocycle and $\delta=\ad U\circ\epsilon$.
Note that
\begin{align*}
\delta(a)&=\ad V(a\xt 1)
\\
\epsilon(a)&=\ad W(a\xt 1).
\end{align*}
We first show that $U$ is an $\epsilon$-cocycle.
For the cocycle identity:
\begin{align*}
    \id\xt\dg(U)
    &=\id\xt\dg(VW^*)
    \\&=\id\xt\dg(V)\id\xt\dg(W)^*
    \\&=V_{12}V_{13}W_{13}^*W_{12}^*
    \\&=V_{12}W_{12}^*
    W_{12}V_{13}W_{13}^*W_{12}^*
    \\&=U_{12}\ad W_{12}(U_{13})
    \\&=U_{12}\epsilon\xt\id(U).
\end{align*}

We still need to show that $U\in M(A\xt C^*(G))$.
By \lemref{fixed} and its proof, we have
\[
A\xt C^*(G)=((A\rtimes_\delta G)\xt C^*(G))^{\alpha\xt\id},
\]
so it suffices to show that if
$c\in C^*(G)$ and $s\in G$
then
\[
(\alpha_s\xt\id)\bigl(U(1\xt c)\bigr)
=U(1\xt c),
\]
and similarly for $(1\xt c)U$.
We only show the first, since
the computation for the second is parallel:
\begin{align*}
    &(\alpha_s\xt\id)\bigl(U(1\xt c)\bigr)
    \\&\qquad=(\alpha_s\xt\id)\bigl((\mu\xt\id)(w_G)
    (\nu\xt\id)(w_G)^*\bigr)(1\xt c)
    \\&\qquad=\mu\xt\id\bigl((\rt_s\xt\id)(w_G)
    \bigr)
    \nu\xt\id\bigl((\rt_s\xt\id)(w_G)\bigr)^*
    (1\xt c)
    \\&\qquad=\mu\xt\id\bigl(w_G(1\xt s)\bigr)
    \nu\xt\id\bigl((1\xt s\inv)w_G^*\bigr)
    (1\xt c)
    \\&\qquad=\mu\xt\id(w_G)
    \nu\xt\id(w_G^*)(1\xt c)
    \\&\qquad=U(1\xt c).
\end{align*}

Finally, if $a\in A$ then
\begin{align*}
    \ad U\circ \epsilon(a)
    &=\ad VW^*\circ\ad W(a\xt 1)
    =\ad V(a\xt 1)
    =\delta(a).
\end{align*}

For the other part, it suffices to note that the above arguments imply that for any such
isomorphism $\Theta$ we have
\begin{align*}
(\Theta\circ j_G^\delta\xt\id)(w_G)
&=\mu\xt\id(w_G)
\\&=V
=VW^*W
=UW
\\&=j_A^\epsilon\xt\id(U)j_G^\epsilon\xt\id(w_G),
\end{align*}
since $U\in M(A\xt\cstg)$ and $j_A^\epsilon=\id_A$.

(2)
The forward direction is \cite[Proposition~2.8 (ii)]{twisted}.
Conversely,  
assume that we have an isomorphism
$\Theta:(A\rtimes_\delta G,\what\delta)
\variso
(A\rtimes_\epsilon G,\what\epsilon)$.
Note that, since by Landstad duality $j_A(A)$ is the fixed-point algebra of $A\rtimes_\delta G$ under the dual action $\what\delta$, and similarly for $j_B(B)$,
it follows from \lemref{fixed} that the isomorphism $\Theta$ takes
$j_A(A)$ to $j_B(B)$.
We dualize the argument of \cite[Theorem~2.1]{rigid}:
Note that $(B\rtimes_\epsilon G,\what\epsilon,j_G^\delta)$
is an equivariant action of $G$, with fixed-point algebra
\[
(B\rtimes_\epsilon G)^{\what\epsilon}
=\Theta\bigl((A\rtimes_\delta G)^{\what\delta}\bigr)
=\Theta(j_A(A))
=j_B(B),
\]
and $j_B:B\to j_B(B)$ is an isomorphism,
so by Landstad duality there are a coaction $\eta$ of $G$ on $B$ and an isomorphism
\[
\psi:(B\rtimes_\eta G,\what\eta)\variso (B\rtimes_\epsilon G,\what\epsilon)
\]
such that
\[
\psi\circ j_G^\eta=j_G^\delta
\midtext{and}
\psi\circ j_B^\eta=j_B^\epsilon.
\]
Thus by part (1) the coactions $\epsilon$ and $\eta$ are exterior equivalent.

On the other hand, we have an isomorphism
\[
\psi\inv\circ\Theta:(A\rtimes_\delta G,\what\delta)\variso (B\rtimes_\eta G,\what\eta),
\]
taking $j_G^\delta$ to $j_G^\eta$,
so again by Landstad duality the coactions $\delta$ and $\eta$ are isomorphic.
Therefore $\delta$ and $\epsilon$ are outer conjugate.
\end{proof}

\section{Category equivalence}

Let $G$ be a fixed compact group.

The \emph{outer category $\coo$ of coactions} is defined as follows:
An object is a coaction of $G$, typically denoted by $(A,\delta)$,
and a morphism in the category is a pair $(\varphi,U): (A,\delta)\to (B,\varepsilon)$,
where $U$ is an $\varepsilon$-cocycle and
$\varphi: A\to M(B)$ is a nondegenerate homomorphism
that is $\delta-\ad U\circ \varepsilon$ equivariant.
In \cite[Lemma~6.8]{three} it is shown that the category $\coo$ is well-defined.
Isomorphisms in the category are outer conjugacies of coactions.

Next, we define the
\emph{category $\AA$ of dual actions},
as in \defnref{dual},
where a morphism
$\psi: (C,\alpha)\to (D,\beta)$
is a nondegenerate 
$\alpha-\beta$ equivariant
homomorphism $\psi: C\to M(D)$.
Note that,
by equivariance,
$\psi$ automatically takes $C^\alpha$ to $M(D^\beta)$.

Suppose $(\varphi,U): (A,\delta)\to (B,\varepsilon)$ in $\coo$.
Then $U$ is an $\varepsilon$-cocycle,
giving an exterior-equivalent coaction $\zeta=\ad U\circ\varepsilon$,
and $\varphi: (A,\delta)\to (B,\zeta)$ is a nondegenerate $\delta - \zeta$ equivariant homomorphism.
Taking crossed products gives a nondegenerate $\what\delta-\what\zeta$ equivariant homomorphism
\[
\varphi\rtimes G: A\rtimes_\delta G\to B\rtimes_\zeta G
\]
taking $j_G^\delta$ to $j_G^\zeta$.

On the other hand, Pedersen's theorem
(\thmref{pedersen}, but here we only need the ``easy half'';
see \cite[Proposition 2.8]{twisted})
gives a
$\what\zeta-\what\varepsilon$ equivariant
isomorphism
\[
\Phi_U: B\rtimes_\zeta G\variso B\rtimes_\varepsilon G,
\]
and then the composition gives a nondegenerate
$\what\delta-\what\varepsilon$ equivariant homomorphism
\[
\Phi_U\circ (\varphi\rtimes G): A\rtimes_\delta G\to B\rtimes_\varepsilon G.
\]

\begin{thm}\label{category equivalence}
With the above notation, the assignments
\begin{align}
(A,\delta)&\mapsto (A\rtimes_\delta G,\what\delta)
\label{obj map}
\\
(\varphi,U)&\mapsto \Phi_U\circ (\varphi\rtimes G)
\label{mor map}
\end{align}
give a category equivalence 
$\wcpo: \coo\to \waco$.
\end{thm}

\begin{proof}
We have to show that \eqref{obj map}--\eqref{mor map} give
a functor
that is faithful, full, and essentially surjective.
It follows immediately from the definitions that the object and morphism maps \eqref{obj map}--\eqref{mor map}
are well-defined and
that \eqref{mor map}
preserves identity morphisms.
To check that $\wcpo$ preserves compositions,
given morphisms
$\xymatrix{
	(A,\delta) \ar[r]^-{(\varphi,U)} &(B,\varepsilon) \ar[r]^-{(\psi,V)} &(C,\zeta)
}$
in $\coo$,
we have (similarly to \cite[Proposition~5.8]{three})
\begin{align*}
\wcpo(\psi,V)\circ \wcpo(\varphi,U)\circ j_A^\delta
&=\bigl(\Phi_V\circ (\psi\rtimes G)\bigr)
\circ
\bigl(\Phi_U\circ (\varphi\rtimes G)\bigr)
\circ j_A^\delta
\\
&=\Phi_V\circ (\psi\rtimes G)\circ \Phi_U\circ j_B^{\ad U\circ\varepsilon}\circ \varphi
\\
&=\Phi_V\circ (\psi\rtimes G)\circ j_B^\varepsilon\circ \varphi
\\
&=\Phi_V\circ j_C^{\ad V\circ\zeta}\circ \psi\circ \varphi
\\
&=j_C^\zeta\circ \psi\circ \varphi
\\
&=\Phi_{(\psi\circ U)V}\circ j_C^{\ad (\psi\circ U)V}\circ \psi\circ\varphi
\\
&=\Phi_{(\psi\circ U)V}\circ \bigl((\psi\circ\varphi)\rtimes G\bigr)\circ j_A^\delta
\\
&=\wcpo\bigl(\psi\circ\varphi,(\psi\circ U)V\bigr)\circ j_A^\delta
\\
&=\wcpo\bigl((\psi,V)\circ (\varphi,U)\bigr)\circ j_A^\delta.
\end{align*}
Moreover 
(exactly as in \cite[Theorem~6.12]{three}),
\begin{align*}
&\bigl(\wcpo(\psi,V)\circ \wcpo(\varphi,U)\circ j_G^\delta\otimes\id\bigr)(w_G)
\\
&\quad=\Bigl(\wcpo\bigl((\psi,V)\circ (\varphi,U)\bigr)\circ j_G^\delta\otimes\id\Bigr)(w_G),
\end{align*}
which gives that
\begin{align*}
\wcpo(\psi,V)\circ \wcpo(\varphi,U)\circ j_G^\delta=\wcpo\bigl((\psi,V)\circ (\varphi,U)\bigr)\circ j_G^\delta.
\end{align*}
It follows that $\wcpo: \coo\to \waco$ is a functor.

To check that $\wcpo$ is essentially surjective,
choose an object $(C,\alpha)$ in $\waco$, so that by definition
there exists a coaction $(A,\delta)$ and an isomorphism
\[
(C,\alpha) \variso (A\rtimes_\delta G,\what\delta).
\]

Finally, we show that for objects $(A,\delta)$ and $(B,\varepsilon)$ in $\coo$,
the functor $\wcpo$ gives a bijection:
\begin{equation}\label{bij}
	\begin{split}
		\mor_{\coo}\bigl((A,\delta),(B,\varepsilon)\bigr)
		& \longleftrightarrow  \mor_{\waco}\bigl(\wcpo(A,\delta),\wcpo(B,\varepsilon)\bigr).
	\end{split}
\end{equation}

First, to see that $\wcpo$ is faithful,
suppose that we have morphisms
\[
(\varphi,U),(\rho,V): (A,\delta)\to (B,\varepsilon)
\]
in $\coo$ such that
\[
\wcpo(\varphi,U)=\wcpo(\rho,V): 
(A\rtimes_\delta G,\what\delta)\to
(B\rtimes_\varepsilon G,\what\varepsilon)
\]
in $\waco$.
By construction, we have
\begin{align*}
\wcpo(\varphi,U)\circ j_A^\delta
&=\Phi_U\circ (\varphi\rtimes G)\circ j_A^\delta
\\
&=\Phi_U\circ j_B^{\ad U\circ\varepsilon}\circ \varphi
\\
&=j_B^\varepsilon\circ \varphi,
\end{align*}
and similarly
\[
\wcpo(\rho,V)\circ j_A^\delta=j_B^\varepsilon\circ \rho,
\]
so $\varphi=\rho$ since $j_B^\varepsilon$ is injective.
Moreover,
\begin{align*}
	\bigl(\wcpo(\varphi,U)\circ j_G^\delta\otimes\id\bigr)(w_G)
	&=\bigl(\Phi_U\circ(\varphi\rtimes G)\circ j_G^\delta\otimes\id\bigr)(w_G)
	\\&=(j_B^\varepsilon\otimes\id)(U)(j_G^\varepsilon\otimes\id)(w_G),
\end{align*}
and similarly
\[
\bigl(\wcpo(\rho,V)\circ j_G^\delta\otimes\id\bigr)(w_G)
=(j_B^\varepsilon\otimes\id)(V)(j_G^\varepsilon\otimes\id)(w_G),
\]
so $(j_B^\varepsilon\otimes\id)(U)=(j_B^\varepsilon\otimes\id)(V)$ since $(j_G^\varepsilon\otimes\id)(w_G)$ is unitary.
Since $j_B^\varepsilon$ is injective, so is
\[
j_B^\varepsilon\otimes\id\colon B\otimes C^*(G)\to M\bigl((B\rtimes_\varepsilon G)\otimes C^*(G)\bigr).
\]
Thus $U=V$, and hence $(\varphi,U)=(\rho,V)$.

Suppose we are given a morphism
\[
\psi\colon (A\rtimes_\delta G,\what\delta)\to (B\rtimes_\varepsilon G,\what\varepsilon)
\]
in $\waco$.
We have that $j_A^\delta(A)=(A\rtimes_\delta G)^{\what\delta}$ and $j_B^\varepsilon(B)=(B\rtimes_\varepsilon G)^{\what\varepsilon}$.
By Landstad duality there are a coaction $(B,\zeta)$ and an isomorphism
\[
\Theta\colon (B\rtimes_\zeta G,\what\zeta)\variso (B\rtimes_\varepsilon G,\what\varepsilon)
\]
taking $j_G^\zeta$ to $\psi\circ j_G^\delta$ such that
\begin{equation}\label{p}
\Theta\circ j_B^\zeta=j_B^\varepsilon.
\end{equation}
In particular, we have a $\what\zeta-\what\varepsilon$ equivariant morphism
$\Theta\colon B\rtimes_\zeta G\to B\rtimes_\varepsilon G$
satisfying \eqref{p}, so by Pedersen's theorem (\thmref{pedersen})
there is a unique $\varepsilon$-cocycle $U$
such that $\zeta=\ad U\circ\varepsilon$ and $\Theta=\Phi_U$.

Furthermore, the map
\[
\Theta\inv\circ \psi: (A\rtimes_\delta G,\what\delta)\to (B\rtimes_\zeta G,\what\zeta)
\]
takes $j_G^\delta$ to $j_G^\zeta$,
so by Landstad duality there exists a map
\[
\varphi: (A,\delta) \to (B,\zeta)
\]
such that $\Theta\inv\circ \psi=\varphi\rtimes G$ and thus
\[
\psi=\Theta\circ (\varphi\rtimes G)=\Phi_U\circ (\varphi\rtimes G)=\wcpo(\varphi,U).
\]
Hence, the map \eqref{bij} is surjective.
\end{proof}

\begin{rem}
Concerning our definition of the functor $\wcpo$, in the discussion preceding \thmref{category equivalence} we only needed the easy half of Pedersen's theorem, which is valid for any locally compact group $G$,
so the assignments \eqref{obj map}--\eqref{mor map}
make sense for arbitrary $G$.
Now, careful inspection of the above proof reveals that the only place where we required $G$ to be compact
and not merely amenable
(note that amenability is required to ensure that $j_B^\epsilon$ is injective)
is the verification that the functor $\wcpo$
is full.
Consequently, if we only assume that $G$ is amenable
we still get an essentially surjective faithful functor $\wcpo:\CC\to \AA$.
This highlights exactly where we need the full force of Pedersen's theorem \thmref{pedersen}.
\end{rem}

Now, let $\mathcal{B}$ be the nondegenerate category of $C^*$-algebras,
that is, the category whose objects are $C^*$-algebras and morphisms are nondegenerate $^*$-homomorphisms.

We define a forgetful functor $F: \waco\to \mathcal{B}$ as usual,
by $(C,\alpha)\to C$ on objects and on morphisms by taking
$(C,\alpha)\to (D,\beta)$
to the map $C\to M(D)$.
We define the \emph{outer crossed-product functor} as the composition
\[
\wilde\wcpo:=F\circ \wcpo: \coo\to \mathcal{B}.
\]
This setting describes an inversion of $\wilde\wcpo$ in the sense of \cite[Section~4]{three} that we call \emph{outer Landstad duality for coactions}.

In \cite[Section~4]{three} we say that the inversion is \emph{good} if the image of $F$ is contained in the essential image of $\wilde\wcpo$,
and $F$ enjoys the following unique isomorphism lifting property:
For all objects $y\in\mathcal{B}$ and $x\in F^{-1}(y)$ and every isomorphism $\theta$ in $\mathcal{B}$ with domain $y$,
there exists a unique isomorphism $\eta$ in $\waco$ such that $F(\eta)=\theta$.

\begin{prop}
Outer Landstad duality for coactions is a good inversion.
\end{prop}	

\begin{proof}
If $C=F(C,\alpha)$
then by definition there is a coaction $(A,\delta)$ such that $C\simeq A\rtimes_\delta G$.
Thus, the image of $F$ is contained in the essential image of $\wilde\wcpo$.

Moreover, given an object $(C,\alpha)$ of $\waco$ and an isomorphism $\theta\colon C\variso D$ in $\mathcal{B}$,
we can use $\theta$ to carry the action $\alpha$
over to an action $\beta$ on $D$,
and then  $\theta$ gives an isomorphism
$\wilde\theta\colon (C,\alpha)\to (D,\beta)$
in $\waco$ covering $\theta\colon C\to D$.
The forgetful functor is faithful, so $\wilde\theta$ is unique.
Thus, the unique isomorphism lifting property is satisfied.
\end{proof}

\providecommand{\bysame}{\leavevmode\hbox to3em{\hrulefill}\thinspace}
\providecommand{\MR}{\relax\ifhmode\unskip\space\fi MR }
\providecommand{\MRhref}[2]{%
  \href{http://www.ams.org/mathscinet-getitem?mr=#1}{#2}
}
\providecommand{\href}[2]{#2}

\end{document}